\documentclass[12pt]{amsart}
\usepackage{a4wide}                
\usepackage{graphicx}
\usepackage{indentfirst}
\usepackage[latin1]{inputenc}
\usepackage{amsmath,amscd}
\usepackage{amsfonts}
\usepackage{amsthm}
\usepackage{graphics}
\usepackage{amssymb}
\usepackage{enumerate}
\usepackage[usenames,dvipsnames]{pstricks}
\usepackage{epsfig}
\usepackage{pst-grad} 
\usepackage{pst-plot} 
\usepackage[all]{xy}

\newtheorem{thm}{Theorem}

\newtheorem{lemma}[thm]{Lemma}
\newtheorem{proposition}[thm]{Proposition}

\theoremstyle{definition}
\newtheorem{definition}[thm]{Definition}
\theoremstyle{remark}
\newtheorem{remark}[thm]{Remark}

\numberwithin{equation}{section}
\numberwithin{thm}{section}

\newcommand\N{\mathbb{N}}
\newcommand\R{\mathbb{R}}
\newcommand\J{\mathcal{J}}

\newcommand\Z{\mathbb{Z}}

\DeclareMathOperator{\Hom}{Hom}

\DeclareMathOperator{\tr}{tr}
\DeclareMathOperator{\SL}{SL}

\newcommand\Gt{\widetilde{G}}

\newcommand\psl{{\rm PSL}(2 , \R)}
\newcommand\pgl{{\rm PGL}(2 , \R)}

\begin{document}

\title[Components of $\Hom (\pi , \psl )$]
{Connected components of spaces of representations of non-orientable surfaces}

\author{Frederic Palesi}
\address{Institut Fourier, \\
100 rue des maths, 38452 St Martin d'Heres}
\email{fpalesi@ujf-grenoble.fr}
\date{\today}

\begin{abstract}
Let $M$ be a compact closed non-orientable surface. We show that the space of representations of the fundamental group of $M$ into $\psl$ has exactly two connected components. These two components are the preimages of a certain Stiefel-Whitney characteristic class, computed in a similar way as the Euler class in the orientable case. 

\end{abstract}

\maketitle
\tableofcontents

\section{Introduction}

Let $M$ be a compact closed surface and  $G$ be a connected semi-simple Lie group.  The space $\Hom (\pi_1 (M) ,G)$ is the space of representations $\pi_1 (M) \rightarrow G$. An important problem is to find its number of connected components and some topological invariants that separate them. When the surface is orientable, this problem was considered and solved for a large variety of groups $G$ (see \cite{bradlow-garcia-gothen} for a summary of known results). In this paper, we will focus on the case of representations of non-orientable surfaces into the group of isometries of the hyperbolic plane.

The obstruction to lifting a representation $\rho : \pi_1 (M) \rightarrow G$ to the universal cover of $G$ is a characteristic class of $\rho$ , which is an element of $H^2 (M, \pi_1 (G))$. This defines a map:
$$o_2 : \Hom (\pi_1(M) , G ) \longrightarrow H^2 (M, \pi_1 (G)). $$
One knows that:
$$ H^2 (M, \pi_1 (G)) = \left\{ \begin{array}{ll} \pi_1 (G) & \mbox{ if M is orientable,} \\ \pi_1 (G) /2\pi_1 (G) & \mbox{ if M is non-orientable,} \end{array} \right. $$
where $2\pi_1 (G)$ is the subgroup of the abelian group $\pi_1 (G)$ of elements of the form $\{ A^2 | A \in \pi_1 (G) \}$. This map provides a useful tool in the study of the connected components of the representation space $\Hom (\pi_1(M) , G )$, because it is continuous and takes values in a discrete set. Hence, it is constant on the connected components of $\Hom (\pi_1 (M) , G)$ and induces a map also denoted $o_2$:
$$o_2 : \pi_0 (\Hom (\pi_1(M) , G )) \longrightarrow H^2 (M, \pi_1 (G)). $$ 

In \cite{goldman}, William Goldman conjectured that for any connected complex semisimple Lie group $G$, and any orientable surface $M$ of genus $g > 1$, the map $o_2$ is a bijection. Jun Li  proved this conjecture in \cite{li}. In the same paper, he also proved that when the group $G$ is connected compact and semisimple, the map $o_2$ is also a bijection. When $M$ is non-orientable, the problem has been studied by Nan-Kuo Ho and Chiu-chu Melissa Liu in \cite{ho-liu} who proved the bijectivity of $o_2$ when the non-orientable genus of $M$ is not $1, 2$ or $ 4$ and $G$ is a connected compact and semisimple Lie group.

These situations contrast with the case when $G$ is a non-compact real form of a semisimple complex Lie group. The most common example is the group $G = \psl$. The fundamental group of $\psl$ is infinite cyclic, $\pi_1 (G) = \Z$. On the other hand, we know that the space of representation of a finitely presented group into an algebraic semisimple Lie group can have only finitely many connected components.

When the group $G$ is $\psl$ and the surface is orientable, the map $o_2$ coincides with a well-known invariant of representations, the Euler class:
$$e : \Hom (\pi , \psl) \longrightarrow \pi_1(\psl) = \Z .$$
For a representation $\phi$, the Euler class satisfies the Milnor-Wood inequality $|e(\phi)| \leq |\chi (M)|$. The main theorem of \cite{goldman} states that the connected components of $\Hom (\pi , G)$ are the preimages $e^{-1} (n)$, where $n$ is any integer satisfying $|n| \leq 2g - 2$. Thus the Euler class distinguishes the $4g-3$ connected components of the space $\Hom (\pi , \psl)$. Generalizations of this result for orientable surfaces have been studied for a variety of other Lie groups, for example  when $G =  \mbox{PSL} (n , \R)$  (see \cite{hitchin}), $G = \mbox{Sp} (2n , \R)$ (see \cite{burger,garcia}), or when the symmetric space associated to $G$ is Hermitian (see \cite{toledo}).

In this paper, our purpose is to generalize the results of Goldman for $G = \psl$ to non-orientable surfaces. When the surface $M$ is non-orientable, the map $o_2$ is defined as a map:
$$o_2 : \Hom (\pi , G) \longrightarrow \pi_1 (G) / 2\pi_1 (G) = \Z / 2 \Z .$$

The main purpose of this paper is to obtain the following:

\begin{thm}\label{thm:closed}
Let $M$ be a closed non-orientable surface of genus $k$ with $k \geq 3$. Then the map 
$$o_2 : \pi_0 (\Hom (\pi_1 (M) , \psl)) \longrightarrow \Z / 2 \Z$$
given by the obstruction class, is a bijection.

In particular, the set $\Hom (\pi_1 (M) , \psl)$ consists of two connected components, which are the preimages $o_2^{-1} (n)$, where $n \in \mathbb{Z} / 2 \mathbb{Z}$.
\end{thm}

This paper is organized as follows.

The construction of the obstruction class $o_2$ for non-orientable surfaces is the object of Section \ref{section:obs}. We also define it for non-orientable surfaces with boundary and state some basic properties. To understand this map, we need to have a full description of the square map in $\SL (2 , \R)$ and its universal cover which can be found in Section \ref{section:square}. The Section \ref{section:path} is devoted to refinements of some of Goldman's intermediary results, to fit with the behavior of the square map. Finally, the Section \ref{section:proof} contains the core of the proof of the main theorem.

\medskip
{\bf Acknowledgments :} The author was partially supported by ANR Repsurf ANR-06-BLAN-0311. The author is grateful to his PhD thesis advisor, Louis Funar, for supervising this work and to Nan-Kuo Ho for useful comments and remarks.

\section{The square map}\label{section:square}

\subsection{Properties of the square map}
For the rest of the paper, let $G$ denote the group of orientation preserving isometries of the hyperbolic plane $\mathbb{H}^2$. We can identify $G$ with the group $\psl = \SL (2 , \R) / \{ \pm I \} $. 

The map
\begin{align*}
Q : \psl & \longrightarrow \SL (2 , \R)\\
    [ A ] & \longmapsto A^2
\end{align*}
is well-defined as $(-A)^2 = A^2$ in $\SL (2 , \R)$. This map satisfies the following:

\begin{proposition}\label{prop:square}
\
\begin{enumerate}[(i)]
\item The image of $Q$ is the set 
$$J = \{ K \in \SL (2 , \R) | \tr (K) > -2 \} \cup \{ -I \} .$$
\item For any $K \in J \setminus \{ -I \}$, there is a unique $A \in \psl$ such that $Q( A) = K$, given by 
 $$A = \frac{K+I}{\sqrt{\tr (K) +2} } .$$
\item The fiber $Q^{-1} ( -I )$ is the set $\displaystyle{\left\{ g \left( \begin{array}{cc} 0 & -1 \\ 1 & 0 \end{array} \right) g^{-1} | \, g \in \psl \right\} }$.
\end{enumerate}
\end{proposition}

\begin{proof}
Let $K$ be an element of $\SL (2 , \R)$, and $A \in \SL (2 , \R)$ such that $A^2 = K$. We have the identity $\tr (A^2) = (\tr (A) )^2 - 2$, hence the trace satisfies $\tr (A^2)  = \tr (K) \geq -2$. 
If $\tr (K) > -2$, then $K$ is conjugated to one of the following matrix:
\begin{itemize}
\item $\displaystyle{ R_{\theta} = \left( \begin{array}{cc} \cos \theta & \sin \theta \\ - \sin \theta & \cos \theta \end{array} \right) }$ with $\displaystyle{\theta \in \R \setminus \{ k\pi | k \in \Z \} }$, if \hbox{$\tr (K) <2$;}
\item $\displaystyle{ H_t = \left( \begin{array}{cc} \cosh t & \sinh t \\ \sinh t & \cosh t \end{array} \right)}$  with $t \in \R^*$, if $\tr (K) >2$;
\item $\displaystyle{ P_u  = \left( \begin{array}{cc} 1 & u \\ 0 & 1 \end{array} \right)}$ with $ u \in \R$, if $\tr (K) =2$.
\end{itemize}

In these three cases, the matricial equation $A^2 = K$ is easily solved. Namely we have: 
\begin{itemize}
\item $A^2 = R_\theta$ if and only if $A = \pm R_{\theta /2}$;
\item $A^2 =P_u$ if and only if $A = \pm P_{u/2}$;
\item $A^2= H_t$ if and only if $A = \pm H_{t/2}$.
\end{itemize}
In each case, one of the solutions satisfies $\tr (A) > 0$, and then we have $\tr (A) = \sqrt{\tr (K) +2}$. The Cayley-Hamilton theorem ensures that $A$ satisfies
 $$A = \frac{K+I}{\sqrt{\tr (K) +2} }.$$
The two solutions in $\SL (2 , \R)$ are identified in $\psl$, and hence $(ii)$ is proved.

If $\tr (K) = -2$, then $K$ is conjugated to the matrix $\displaystyle{\left( \begin{array}{cc} -1 & u \\ 0 & -1 \end{array} \right)}$. 

If $u \neq 0$, the equation $A^2 = K$ has no solutions in $\SL (2 , \R)$. 

If $u=0$, we have $K = -I$ and then the matrix $\displaystyle{A = \left( \begin{array}{cc} x & y \\ z & t \end{array} \right) \in \SL (2 , \R) }$ is a solution, if and only if the coefficients satisfy $z\neq 0$, $t = -x$ and $y = \frac{-1 - x^2}{z}$. Hence the set of solutions in $\SL (2 , \R) $ is
$$\left\{ \left( \begin{array}{cc} x & \frac{-1-x^2}{z} \\ z & -x \end{array} \right) | x\in \R , z \in \R^* \right\} \subset \SL (2 , \R). $$
Each solution in $\psl$ has a couple of representatives of the form $(A , -A)$ that are solutions in $\SL (2 , \R)$. We choose the representative satisfying $z > 0$ and define the matrix 
$$g = \left( \begin{array}{cc} \sqrt{z} & x/\sqrt{z} \\ & \\ 0 & 1/\sqrt{z} \end{array} \right) \in  \SL (2 , \R).$$
We see that
$$g \left( \begin{array}{cc} 0 & -1 \\ 1 & 0 \end{array} \right) g^{-1} = \left( \begin{array}{cc} x & \frac{-1-x^2}{z} \\ z & -x \end{array} \right) = A .$$
Conversely, any matrix $A \in \psl$ of the form $\displaystyle{A = g \left( \begin{array}{cc} 0 & -1 \\ 1 & 0 \end{array} \right) g^{-1}}$ satisfies \hbox{$Q (A) = - I$.}
\end{proof}

\subsection{The square map of the universal cover}

Let $\Gt$ be the universal cover of $G$ with $p : \Gt \rightarrow G$ the covering map. The group $G$ is \hbox{homeomorphic} to a solid torus, and hence $\pi_1 (G) = \Z$. The center of $\Gt$ is isomorphic to $\mathbb{Z}$ and $G = \Gt / Z (\Gt)$. We denote by $z$ a generator of $Z (\Gt)$, so that $Z (\Gt) = \langle z \rangle$.

The group $\SL (2 , \R)$ is a connected 2-fold cover of $\psl$, and thus $\Gt$ is also the universal cover of $\SL (2 , \R)$. We have a canonical identification 
$$ \SL (2 , \R) = \Gt / \langle z^2 \rangle . $$
For $A \in \Gt$ we define $\tr (A)$ to be the trace of the image of $A$ in $\SL (2 , \R)$. We say that $A \notin Z (\Gt)$ is {\it elliptic, parabolic} or {\it hyperbolic} according as $|\tr (A)| <2 , |\tr (A)| = 2$ or $|\tr (A)| >2$ and we denote by $\mathcal{E}$, $\mathcal{P}$ and $\mathcal{H}$ the subsets of $\Gt$ consisting of elliptic, parabolic and hyperbolic elements. These subsets decompose $\Gt$ into infinitely many components, indexed by $\Z$, according to Figure \ref{domaines}.

\begin{figure}[ht]
\begin{center}
\scalebox{0.9} 

\begin{pspicture}(0,-1.9812986)(13.62,1.9412986)
\psline[linewidth=0.04cm](1.02,1.9112986)(12.64,1.9112986)
\psline[linewidth=0.04cm](1.0,-1.3087014)(12.62,-1.3087014)
\psline[linewidth=0.04cm,linestyle=dashed,dash=0.16cm 0.16cm](1.02,1.9112986)(0.02,1.9112986)
\psline[linewidth=0.04cm,linestyle=dashed,dash=0.16cm 0.16cm](1.0,-1.3087014)(0.0,-1.3087014)
\psline[linewidth=0.04cm,linestyle=dashed,dash=0.16cm 0.16cm](13.6,1.9112986)(12.6,1.9112986)
\psline[linewidth=0.04cm,linestyle=dashed,dash=0.16cm 0.16cm](13.6,-1.3087014)(12.6,-1.3087014)
\psline[linewidth=0.06cm](2.6,1.9112986)(5.8,-1.3087014)
\psline[linewidth=0.06cm](5.8,1.9112986)(2.6,-1.3087014)
\psbezier[linewidth=0.03](5.8,1.9112986)(5.38,1.4912986)(5.38,-0.8687014)(5.82,-1.2887014)
\psbezier[linewidth=0.03,linestyle=dotted,dotsep=0.16cm](5.8,1.9112986)(6.2,1.5112987)(6.2,-0.8887014)(5.82,-1.2887014)
\psline[linewidth=0.06cm](5.8,1.8912987)(9.0,-1.3287014)
\psline[linewidth=0.06cm](9.0,1.8912987)(5.8,-1.3287014)
\psbezier[linewidth=0.03](9.0,1.8912987)(8.58,1.4712986)(8.58,-0.8887014)(9.02,-1.3087014)
\psbezier[linewidth=0.03,linestyle=dotted,dotsep=0.16cm](9.0,1.8912987)(9.4,1.4912986)(9.4,-0.90870136)(9.02,-1.3087014)
\psline[linewidth=0.06cm](9.0,1.9112986)(12.2,-1.3087014)
\psline[linewidth=0.06cm](12.2,1.9112986)(9.0,-1.3087014)
\psbezier[linewidth=0.03](12.2,1.9112986)(11.78,1.4912986)(11.78,-0.8687014)(12.22,-1.2887014)
\psbezier[linewidth=0.03,linestyle=dotted,dotsep=0.16cm](12.2,1.9112986)(12.6,1.5112987)(12.6,-0.8887014)(12.22,-1.2887014)
\psline[linewidth=0.06cm](1.04,0.2912986)(2.6,-1.3087014)
\psline[linewidth=0.06cm](2.6,1.9112986)(1.0,0.2912986)
\psbezier[linewidth=0.03](2.6,1.9112986)(2.18,1.4912986)(2.18,-0.8687014)(2.62,-1.2887014)
\psbezier[linewidth=0.03,linestyle=dotted,dotsep=0.16cm](2.6,1.9112986)(3.0,1.5112987)(3.0,-0.8887014)(2.62,-1.2887014)
\usefont{T1}{ptm}{m}{n}
\rput(9.209844,0.3412986){\psframebox[linewidth=0.028222222,linecolor=white,fillstyle=solid,framesep=0.0,boxsep=false]{$\mathcal{E}_0$}}
\usefont{T1}{ptm}{m}{n}
\rput(5.994219,0.30129862){\psframebox[linewidth=0.028222222,linecolor=white,fillstyle=solid,framesep=0.0,boxsep=false]{$\mathcal{E}_{-1}$}}
\usefont{T1}{ptm}{m}{n}
\rput(2.76875,0.3212986){\psframebox[linewidth=0.028222222,linecolor=white,fillstyle=solid,framesep=0.0,boxsep=false]{$\mathcal{E}_{-2}$}}
\usefont{T1}{ptm}{m}{n}
\rput(12.474218,0.36129862){\psframebox[linewidth=0.028222222,linecolor=white,fillstyle=solid,framesep=0.0,boxsep=false]{$\mathcal{E}_1$}}
\usefont{T1}{ptm}{m}{n}
\rput(7.2704687,1.3412986){\psframebox[linewidth=0.028222222,linecolor=white,fillstyle=solid,framesep=0.0,boxsep=false]{$\mathcal{H}_0$}}
\usefont{T1}{ptm}{m}{n}
\rput(10.574843,1.3412986){\psframebox[linewidth=0.028222222,linecolor=white,fillstyle=solid,framesep=0.0,boxsep=false]{$\mathcal{H}_1$}}
\usefont{T1}{ptm}{m}{n}
\rput(4.194844,1.3412986){\psframebox[linewidth=0.028222222,linecolor=white,fillstyle=solid,framesep=0.0,boxsep=false]{$\mathcal{H}_{-1}$}}
\psarc[linewidth=0.04](3.12,-1.1087013){0.68}{317.48956}{59.82648}
\rput{42.969494}(0.5609299,-2.5740309){\pstriangle[linewidth=0.04,dimen=outer](3.5502963,-0.73102844)(0.2088893,0.31314206)}
\psarc[linewidth=0.04](5.41,-1.1187013){0.75}{127.40536}{214.69516}
\rput{-44.193214}(1.7927115,3.2267516){\pstriangle[linewidth=0.04,dimen=outer](4.8702965,-0.7510284)(0.2088893,0.31314206)}
\usefont{T1}{ptm}{m}{n}
\rput(3.6698437,-1.7787014){\psframebox[linewidth=0.028222222,linecolor=white,fillstyle=solid,framesep=0.0,boxsep=false]{$\mathcal{P}_{-1}^-$}}
\usefont{T1}{ptm}{m}{n}
\rput(4.975,-1.7787014){\psframebox[linewidth=0.028222222,linecolor=white,fillstyle=solid,framesep=0.0,boxsep=false]{$\mathcal{P}_{-1}^+$}}
\psarc[linewidth=0.04](6.32,-1.1487014){0.68}{317.48956}{59.82648}
\rput{44.780567}(1.5260508,-4.9331746){\pstriangle[linewidth=0.04,dimen=outer](6.750296,-0.77102846)(0.2088893,0.31314206)}
\psarc[linewidth=0.04](8.61,-1.1587014){0.75}{127.40536}{214.69516}
\rput{-42.05164}(2.5027263,5.242139){\pstriangle[linewidth=0.04,dimen=outer](8.070296,-0.79102844)(0.2088893,0.31314206)}
\usefont{T1}{ptm}{m}{n}
\rput(6.7898436,-1.7987014){\psframebox[linewidth=0.028222222,linecolor=white,fillstyle=solid,framesep=0.0,boxsep=false]{$\mathcal{P}_{0}^-$}}
\usefont{T1}{ptm}{m}{n}
\rput(8.175,-1.7787014){\psframebox[linewidth=0.028222222,linecolor=white,fillstyle=solid,framesep=0.0,boxsep=false]{$\mathcal{P}_{0}^+$}}
\usefont{T1}{ptm}{m}{n}
\rput(10.029843,-1.7787014){\psframebox[linewidth=0.028222222,linecolor=white,fillstyle=solid,framesep=0.0,boxsep=false]{$\mathcal{P}_{1}^-$}}
\usefont{T1}{ptm}{m}{n}
\rput(11.355,-1.7987014){\psframebox[linewidth=0.028222222,linecolor=white,fillstyle=solid,framesep=0.0,boxsep=false]{$\mathcal{P}_{1}^+$}}
\psarc[linewidth=0.04](9.5,-1.1087013){0.68}{317.48956}{59.82648}
\rput{42.653996}(2.237737,-6.880432){\pstriangle[linewidth=0.04,dimen=outer](9.930296,-0.73102844)(0.2088893,0.31314206)}
\psarc[linewidth=0.04](11.79,-1.1187013){0.75}{127.40536}{214.69516}
\rput{-42.637463}(3.376627,7.4633265){\pstriangle[linewidth=0.04,dimen=outer](11.250297,-0.7510284)(0.2088893,0.31314206)}
\psdots[dotsize=0.15](4.2,0.3112986)
\psdots[dotsize=0.15](7.4,0.2912986)
\psdots[dotsize=0.15](10.6,0.3112986)
\usefont{T1}{ptm}{m}{n}
\rput(4.7834377,0.3212986){\psframebox[linewidth=0.028222222,linecolor=white,fillstyle=solid,framesep=0.0,boxsep=false]{$z^{-1}$}}
\usefont{T1}{ptm}{m}{n}
\rput(7.9334375,0.3212986){\psframebox[linewidth=0.028222222,linecolor=white,fillstyle=solid,framesep=0.0,boxsep=false]{$1$}}
\usefont{T1}{ptm}{m}{n}
\rput(11.093437,0.3212986){\psframebox[linewidth=0.028222222,linecolor=white,fillstyle=solid,framesep=0.0,boxsep=false]{$z^1$}}
\end{pspicture} 

\end{center}
\caption{Domains of $\Gt$}\label{domaines}
\end{figure}

We can distinguish these regions by the following invariants of $A \in \Gt$. An element $A \in \psl$ acts by projective automorphisms on the boundary $\partial \mathbb{H}^2 \simeq \mathbb{S}^1$. This action lifts to an action of $\Gt$ on $\widetilde{\mathbb{S}^1} = \R$. For $A \in \Gt$ we define
\begin{align*}
\underline{m} A &= \min \{ A \cdot x - x \, | \, x \in \R \} ,\\
\overline{m} A & = \max \{ A \cdot x - x \, | \, x \in \R \} .
\end{align*}

\begin{lemma}[\cite{jankins}] We have
\begin{align*}
A \in \mathcal{E}_i & \mbox{ if and only if } [\underline{m} A, \overline{m} A ] \subset ]i, i+ 1[, \\
A \in \mathcal{H}_i & \mbox{ if and only if } i \in \, ] \underline{m} A, \overline{m} A[, \\
A \in \mathcal{P}_i^+ & \mbox{ if and only if } \underline{m} A = i < \overline{m} A, \\
A \in \mathcal{P}_i^- & \mbox{ if and only if } \underline{m} A < i = \overline{m} A.  
\end{align*}
\end{lemma}

Let $\J \subset \Gt$ be the set of elements of $\Gt$ whose image in $\SL (2 , \R)$ is in $J$, or equivalently the image of the map
\begin{align*}
\widetilde{Q} : \Gt & \longrightarrow \Gt \\
A & \longmapsto A^2 .
\end{align*}
We have
$$\J = \left\{ A \in \Gt | \tr (A) > -2 \right\} \cup \left\{ z^{2k +1} | k\in \Z \right\} $$
For an element $K \in \Gt$, we have $\tr (K) < - 2$ if and only if $K \in  \mathcal{H}_{2k+1}$ for a certain $k \in \Z$. Likewise, $\tr (K) = -2$ if and only if \hbox{$K \in  \mathcal{P}_{2k+1} \cup \{ z^{2k+1} \}$.} We infer that 
$$\J = \Gt \setminus \left( \bigcup_{k\in \Z} \mathcal{H}_{2k+1} \cup \mathcal{P}_{2k+1} \right) .$$

As a consequence of Proposition \ref{prop:square}, we can state a path-lifting \hbox{property} of the square map:

\begin{proposition}\label{prop:liftsquare} 
Let $\{ K \}_{t \in [0,1]}$ be a path in $\Gt$ satisfying the following properties:
\begin{enumerate}
\item For all t in $[0 , 1]$, the element $K_t$ is in $\J$.
\item The set $\mathcal{T} = \{ t \in ]0 , 1[ \, | \, K_t \neq z^{2k+1}$ for any $k \in \mathbb{Z} \}$ is a finite union of open intervals.
\item For every $s \in [0,1]$ such that $K_s = z^{2k+1}$, there exists $\epsilon > 0$, $h_s$ and $g_s$ in $\Gt$ such that : 
$$\forall t \in ( s - \epsilon , s ), K_t = g_s \widetilde{R}_{\theta_t} g_s^{-1}$$ 
$$\forall t \in ( s, s + \epsilon ),  K_t = h_s \widetilde{R}_{\theta_t} h_s^{-1}$$
 where $\widetilde{R}_{\theta_t}$ is a lift of $R_{\theta_t}$ with $\theta_t$ converging towards $\pi$.
\end{enumerate}
Then, up to a reparametrization of the path $\{ K_t \}_{t\in [0,1]}$, there exists a continuous path $\{A_t \}_{t\in[0,1}$ such that for all $t \in [0,1]$, we have $\widetilde{Q} (A_t) = K_t $.
\end{proposition}

\begin{proof}
Let $\{ K_t \}_{t\in [0,1]}$ be a path satisfying the hypotheses. The set $\mathcal{T}$ is a finite union of intervals denoted by $T_i = ] s_i^- , s_i^+[ $, with $i \in \{0, m \}$. On each interval $T_i$, the path in $\Gt$ defined by
\begin{equation}\label{AT}
A_t= \dfrac{K_t + I}{\sqrt{\tr (K_t) +2} }, \, \forall t \in T_i
\end{equation}
is continuous. 

Let $i\in \{ 0 , 1 , \dots  , m \}$ such that $s_i^- \neq 0$, and $k \in \Z$ such that $K_s = z^{2k+1}$. On a right neighborhood of the point $s_i^- $,  we can write $K_t = g_i R_{\theta_t} g_i^{-1}$ with $t\in \mathcal{T}$. As $\displaystyle{\lim_{t\rightarrow s_i^- } K_t = z^{2k+1} }$, we infer that $\displaystyle{ \lim_{t\rightarrow s_i^- } \theta_t = \pi + 2k \pi }$. On this neighborhood, the matrix $A_t$ is defined (\ref{AT}) which gives us:
$$A_t = g_i \widetilde{R}_{\theta_t /2} g_i^{-1}.$$ 
Hence, the path $\{ A_t \}_{t\in T_i}$ has a limit when $t$ converges towards $s_i^-$ with $t > s_i^-$. This limit is the matrix in $\Gt$ defined by
$$A_{i}^- = \lim_{t\stackrel{>}{\rightarrow} s_i^-} A_t =  g_i \widetilde{R}_{\pi/2} g_i^{-1} .$$

Using the same argument in a left neighborhood of the point $s_i^+$, we define the matrix $A_i^+$ in $\Gt$ by:
$$A_{i}^- = \lim_{t\stackrel{>}{\rightarrow} s_i^+} A_t =  h_i \widetilde{R}_{\pi/2} h_i^{-1} .$$

Hence, on each closed interval $\bar{T_i}$, we have a path $\{ A_t \}_{t\in \bar{T_i}}$ satisfying the desired properties. It suffices to find paths on the intervals of $]0,1[ \setminus \mathcal{T}$ such that the endpoints of these paths coincide.

Up to a possible reparametrization of the path $\{ K_t \}_{t\in [0,1]}$, we can assume that $s_i^+ \neq s_{i+1}^-$. For all $t \in [s_i^+ , s_{i+1}^-] = T_i '$, we have $K_t = z^{2k+1}$. The group $G$ being connected, there exists a path joining $h_i$ to $g_{i+1}$. Hence, there exists a path $\{ A_t \}_{t\in T_i '}$ joining $A_i^+$ to $A_{i+1}^-$ in $Q^{-1} (I)$. 

The paths $\{ A_t \}_{ t \in T_i }$, $\{ A_t \}_{t \in T_i '}$ and $\{ A_t \}_{t \in T_{i+1} }$ are compatible for gluing, in the sense that the limits at the common boundary points are equals. Hence the path  $\{ A_t \}_{t \in [0,1]}$ is continuous and satisfies the desired properties.
\end{proof}

\begin{remark}
Without the condition $(3)$, the Proposition \ref{prop:liftsquare} is false. For $t \neq 0$, consider the matrix in $\SL (2 , \R)$ defined by
$$g_t = \left( \begin{array}{cc} \sqrt{2} + \sin \left(\frac{1}{t}\right) & \cos \left(\frac{1}{t}\right) \\ \left(\frac{1}{t}\right) & \sqrt{2} - \sin \left(\frac{1}{t}\right) \end{array} \right)$$
and let $\theta_t = \pi - t$. Consider the path in $\SL (2 , \R)$ given by
$$K_t = \left\{ \begin{array}{ll} g_t R_{\theta_t} g_t^{-1} & \mbox{ if }t \neq 0\\ -I &\mbox{ if } t = 0. \end{array} \right. $$
This path is continuous and lies within $\mathcal{J}$. For all $t \neq 0$, there is a unique $A_t \in \psl$ such that $Q (A_t) = K_t$ given by 
$$A_t = g_t R_{(\theta_t / 2)} g_t^{-1}.$$
However, $A_t$ does not converge when $t$ goes to $0$. Hence the path $\{ K_t \}_{t\in [0,1]}$ can not be lifted in $\psl$.
\end{remark}


\section{The obstruction map}\label{section:obs}
In this section, we define and compute the map $o_2$ in the context of non-orientable surfaces, and we give some of its properties.

\subsection{Definitions}
Let $M$ be a closed non-orientable surface of genus $k$, with $k\geq 1$. Its fundamental group admits the following presentation:
$$\pi_1 (M) = \pi = \langle A_1 , \dots , A_k | A_1^2 \cdots A_k^2 \rangle .$$
The obstruction map
$$o_2  : \Hom (\pi , G) \rightarrow  \Z / 2\Z ,$$
is defined as follows.

Let $\phi$ be an element of $\Hom (\pi, G)$. Choose lifts $\widetilde{\phi (A_1)} , \dots , \widetilde{\phi(A_k)} $ of the images of the generators into $\widetilde{G}$. The relation of the fundamental group implies that the element 
$$(\widetilde{\phi (A_1)})^2 \cdots (\widetilde{\phi (A_k)})^2$$
 is a lift of the identity element of $G$ into $\Gt$. Hence, it is an element of $\pi_1 (G) \cong \Z $. This element is not independent of the chosen lift (contrary to the analogous situation in orientable surface). However the element is well-defined up to an element of $2 \pi_1 (G) \cong 2 \Z$. This gives a well-defined element of $ \Z / 2\Z$, that we denote by $o_2 (\phi)$ and we call the Stiefel-Whitney class of $\phi$.

The map $o_2$ is continuous and hence allows us to partition the space $\Hom (\pi , G)$ into sets that are both open and closed. We can now recall the Theorem \ref{thm:closed} as:
\begin{thm}
Let $M$ be a closed non-orientable surface of genus $k$ with $k \geq 3$. The set $\Hom (\pi_1 (M) , \psl )$ has two connected components, which are given by the preimages $o_2^{-1} (n)$, where $n$ is an element of $\mathbb{Z} / 2 \mathbb{Z}$.
\end{thm}

To prove this theorem, we have to decompose the surface $M$ into surfaces that are necessarily surfaces with boundary.

\subsection{Surfaces with boundary}

In this section, we state a generalization of Theorem \ref{thm:closed} to surfaces with boundary. 

Let $M$ be a non-orientable surface of genus $k$ with $m$ boundary components. Its fundamental group $\pi$ admits the following presentation: 
$$\pi =  \langle A_1 , \dots , A_k , C_1 , \dots , C_m | A_1^2 \cdots A_k^2 \cdot C_1 \cdots C_m\rangle$$
where $C_1 , \dots , C_m$ correspond to the components of $\partial M$.  The group $\pi$ is isomorphic to the free group in $k+m-1$ generators, and the representation space $\Hom (\pi_1 (M) , G)$ is identified with $G^{k+m-1}$. As the group $G = \psl$ is connected, the representation space is also connected. We shall need to impose boundary conditions in order to define a suitable obstruction map.

Suppose that an element $\phi$ of $\Hom (\pi , G)$ is a homomorphism such that for each boundary component $C_i$, the image $\phi (C_i)$ is hyperbolic. Let $W(M)$ denote the set of all homomorphisms satisfying this condition. For any element $g \in \mathcal{H}$ there exists a unique lift $\widetilde{g}^0$ in $\mathcal{H}_0$. Let $\widetilde{\phi (C_i)}^0$ be the lift of $\phi (C_i)$ into $\mathcal{H}_0$ for any $i \in \{ 1 , \dots , m \}$. Choose lifts $\widetilde{\phi (A_1)} , \dots , \widetilde{\phi(A_k)} $ of the images of the generators into $\Gt$. The relation of the fundamental group implies that the element 
$$(\widetilde{\phi (A_1)})^2 \cdots (\widetilde{\phi (A_k)})^2 \widetilde{\phi (C_1)}^0 \dots \widetilde{\phi (C_m)}^0$$
 is a lift of the identity element of $G$ into $\Gt$. As in the closed case, this gives a well-defined element of $ \Z / 2\Z$, that we call the {\it relative Stiefel-Whitney class} of $\phi$, and also denote $o_2 (\phi)$. This gives the map
 $$o_2 : W(M) \longrightarrow \Z/ 2 \Z .$$
The generalization of Theorem \ref{thm:closed} is the following:

\begin{thm}\label{thm:open}
Let $M$ be a compact non-orientable surface with \hbox{$\chi (M)\leq -1$.} The set $W(M)$ has two connected components, which are given by the preimages $o_2^{-1} (n)$ where $n$ is an element of $\mathbb{Z} / 2 \mathbb{Z}$.
\end{thm}

The rest of the paper will be devoted to the proof of this theorem.
\\

For an orientable surface with boundary $S$, we denote by 
$$e : W(S) \rightarrow \Z$$
the relative Euler class as defined by Goldman in \cite{goldman}. The relative Stiefel-Whitney class enjoys a simple additivity formula, similar to the orientable case.

\begin{proposition} Let $M$ be a non-orientable surface such that $M = M_1 \cup M_2$ and $\phi \in \Hom(\pi , G)$ such that for each boundary component $C$ of $M_i$, the restriction of $\phi$ to $\pi_1 (C)$ is hyperbolic. Without loss of generality we may assume that $M_1$ is non-orientable.
Then,
\begin{itemize}
\item if $M_2$ is non-orientable
$$o_2 (\phi) = o_2 (\phi \mid_{\pi_1 (M_1)}) + o_2 (\phi \mid_{\pi_1 (M_2)}) \in \Z / 2 \Z.$$
\item If $M_2$ is orientable
$$o_2 (\phi) = o_2 (\phi \mid_{\pi_1 (M_1)}) + \bar{e} (\phi \mid_{\pi_1 (M_2)})   \in \Z / 2 \Z.$$
\end{itemize}
where $\bar{e}(\phi)  = e (\phi) (\mbox{ mod } 2) \in \Z / 2 \Z$
\end{proposition}


\section{Path-lifting properties}\label{section:path}

In this section, we construct paths in the representation space of an orientable surface $\Sigma$ with non-empty boundary, that join two different components of $W (\Sigma)$. We want to be able to extend a representation of the orientable surface, to a representation of the non-orientable surface obtained by gluing a M\"obius strip along one of the boundary component. 

\subsection{Path-lifting in hyperbolic case}
In \cite{goldman}, Goldman stated the following:

\begin{lemma}\label{lem:releve-1}
Let $\Sigma$ be a three-holed sphere or a one-holed torus, and $C$ denote a boundary component of $\Sigma$. The evaluation map 
\begin{align*}
ev_C : W (\Sigma) & \longrightarrow \mathcal{H} \\
\phi & \longmapsto \phi (C)
\end{align*}
satisfies the path-lifting property.
\end{lemma}

We recall the following definition:

\begin{definition} 
A map $f : X \longrightarrow Y$ satisfies the {\it path-lifting property} if for every $x \in X$ and a path $\{ y_t \}_{0\leq t \leq 1}$ with $f (x) = y_0$, there exists a nondecreasing surjective map (a reparametrization of the path) $\tau : [0 , 1] \rightarrow [0,1]$ and a path $\{ x_s \}_{0\leq s \leq 1}$ such that $f(x_s) = y_{\tau (s)}$ and $x_0 = x$.
\end{definition}

We can generalize Lemma \ref{lem:releve-1} to any orientable surface with non-empty boundary.

\begin{proposition}\label{lem:releve}
Let $\Sigma$ be an orientable surface with non-empty boundary $C_1 , \dots , C_m \subset \partial \Sigma$. Let $C = C_1$ be one of the boundary components. The map
\begin{align*}
ev_C  : W(\Sigma) &\longrightarrow \mathcal{H}\\
\phi & \longmapsto \phi (C)
\end{align*}
satisfies the path-lifting property.
\end{proposition}

\begin{proof}
We prove this result by induction on the Euler characteristic. Lemma \ref{lem:releve-1} shows that the result holds if $\chi(\Sigma) = -1$. For a surface $\Sigma$ with $\chi(\Sigma) \leq -2$,  let $\displaystyle{\Sigma = \bigcup_{i=1}^{-\chi(\Sigma)} \Sigma_i}$ be a decomposition into pair-of-pants, such that $C$ is a boundary component of $\Sigma_1$, and if we denote the two other boundary components of $\Sigma_1$ by $D_1$ and $D_2$ and by $S_i$ the connected component of $\Sigma \setminus \Sigma_1$ that contains $D_i$ as a boundary component (the surface $S_i$ can be empty), then $S_1 \neq S_2$.

Let $\{ \gamma_t \}$ be a path in $\mathcal{H}$, and $\phi_0 $ a representation in $W(\Sigma)$ with $\gamma_0 = \phi_0 (C)$. In each connected component of $ev_C^{-1} (\gamma_0)$, there is a representation $\phi$ such that $\phi (D_1)$ and $\phi (D_2)$ are hyperbolic. Hence, we can assume without loss of generality that $\phi_0 (D_1)$ and $\phi_0 (D_2)$ are hyperbolic. By the path lifting property of Lemma \ref{lem:releve-1}, we can find a path of representations $\psi_t $ in $W (\Sigma_1)$, such that $\psi_t (C) =  \gamma_t$ and $\psi_0 = \phi_{| \pi_1 (\Sigma_1)}$. The paths $\psi_t (D_1)$ and $\psi_t  (D_2)$ are paths in $\mathcal{H}$. When $S_i$ is non-empty, we have $\chi (S_i) \geq \chi (\Sigma) +1$, hence we can apply the induction hypotheses to lift these paths to paths of representations $\{ \psi_t^{(i)} \}$ in $W ( S_i)$.

The path of representations $\{ \phi_t \}$ defined by
$$
\phi_t (\alpha ) = \left\{ \begin{array}{ll} \psi_t  (\alpha ) & \mbox{ if } \alpha \in \pi_1 (\Sigma_1), \\ \psi_t^{(j)} (\alpha) & \mbox{ if } \alpha \in \pi_1(S_j)\end{array} \right.
$$
satisfies $\phi_t (C) = \gamma_t$, and for all $1 \leq j \leq m$, the element $\phi_t (C_j)$ is in $\mathcal{H}$ for all $t \in [0,1]$. Thus, we have a path in $W(\Sigma)$ with the desired property.
\end{proof}

\subsection{Compatibility with $Q$}

Our objective is to find particular paths in the representation space of orientable surfaces, such that the path corresponding to the evaluation of the representation at one of the boundary component, satisfies the properties of Proposition \ref{prop:liftsquare}.

First, let $\Sigma$ be a three-holed sphere. Its fundamental group is free of rank two. Hence a representation in $\Hom(\pi, G)$ is determined by the image of the two generators, and $\Hom (\pi , G)$ identifies with $G \times G$. To find paths of representation in $\Hom (\pi , G)$ it is sufficient to find  paths in $G \times G$. As $G = \psl$ is covered by  $\SL (2 , \R)$, it is sufficient to find paths in $\SL (2 , \R) \times \SL (2 , \R) $. 

The character map is defined by:
$$\begin{array}{rccc}
\chi : &\SL (2 , \R) \times \SL (2 , \R) & \longrightarrow &\R^3 \\
& (X  , Y) &\longmapsto  & \left( \begin{array}{c} \tr (X) \\ \tr (Y) \\ \tr (XY) \end{array} \right) .
\end{array}$$
  
This map and the following lemma of Goldman (\cite{goldman} Corollary 4.5) will  be useful in the sequel. 

\begin{lemma}\label{lem:pathliftchi}
Let $\kappa (x,y,z) = x^2 + y^2 + z^2 - xyz - 2$ and
$$\Omega_{\R} = \{ (X , Y ) \in \SL (2 , \R) \times \SL (2 , \R) \, | \, [ X , Y ] \neq I \} .$$
Then the character map
$$\chi : \Omega_{\R} \longrightarrow \R^3 \setminus [-2 , 2]^3 \cap \kappa^{-1} ([-2 , 2])$$
satisfies the path-lifting property.
\end{lemma}

In particular, this Lemma is a key ingredient of the following proposition:

\begin{proposition}\label{lem:releveJ}
Let $\Sigma$ be a three-holed sphere, with boundary components $B , C , K$. Let $\phi$ be a representation in $W(\Sigma)$ such that $e( \phi) = -1$. There exists a path of representations $\{ \phi_t \}_{0\leq t\leq 1}$ such that :
\begin{enumerate}
\item $\phi_0 = \phi$;
\item $e (\phi_1) = 1$;
\item The elements $\phi_t (B)$ and $\phi_t (C)$ are in $\mathcal{H}$ for all $t \in [0 , 1]$;
\item If $\{ \widetilde{\phi_t} \}_{t\in [0,1]}$ is a lift of the path in $\Hom (\pi, \Gt)$ such that \hbox{$\widetilde{\phi_0} (K) \in z^\epsilon \mathcal{J}$},  then the path $\{ z^\epsilon  \widetilde{\phi_t} (K) \}$ satisfies the hypotheses of Proposition \ref{prop:liftsquare}.
\end{enumerate} 
\end{proposition}

\begin{proof}
Let $\phi \in W (\Sigma)$ such that $e(\phi) = -1$, and let $\widetilde{\phi (B)}^0 , \widetilde{\phi (C)}^0$ be the unique lifts in $\mathcal{H}_0$. The element $\widetilde{\phi (B)}^0 \widetilde{\phi (C)}^0$ is in $\mathcal{H}_{-1}$, and hence we have $\chi (\widetilde{\phi (B)}^0 , \widetilde{\phi (C)}^0 )= (b,c,k)$ with $b, c \in   ]2 , +\infty[$ and $k \in ]-\infty , -2 [.$ Note that $\kappa (b,c,k) \neq 2$.

Let $g$ be an element of $\psl$ such that, for $\theta \in ] 0, \pi [ $, the elliptic element $R_{\theta} ' = g R_\theta g^{-1}$ satisfies $[ \phi (B) , R_\theta '] \neq I $. For $0 \leq \theta \leq \pi$, we take the lift of $R_{\theta}'$ in $\Gt$ belonging to $\mathcal{E}_0$, and we denote it also $R_\theta '$. The path $C_t =  (\widetilde{\phi (B)}^0)^{-1} R_t '$ in $\Gt$ is continuous and starts from $\mathcal{H}_0$. Hence, as $\mathcal{H}_0$ is an open set, there exists $s \in ] 0 , 1 [$ such that $C_t$ is in $\mathcal{H}_{0}$ for all $t \leq s$. Let $B_t = \widetilde{\phi (B) }^0$ in $\Gt$ for all $t \leq s$, so the element $B_t C_t = R_t '$ is in $\mathcal{E}_{0}$. Thus, we have $\chi (B_s , C_s) = ( b , c' , k')$, with $b>2$, $c' > 2$ and $k' \in ]-2 , 2 [$.

The path in $\R^3$ defined by:
\begin{align*}
b_t & = b ,\\
c_t & = ( 1 - \frac{t-s}{1-s}) c' +\frac{t-s}{1-s} c , \\
 k_t & = ( 1 - \frac{t-s}{1-s}) k' +\frac{t-s}{1-s} k,
\end{align*}
never meets the set $[-2 , 2]^3 \cap \kappa^{-1} ([-2 , 2])$. We also have $[ B_s , C_s ] = [R_t ' , (\phi (B))^{-1} ] \neq I$ and according to the Lemma \ref{lem:pathliftchi}, the path can be lifted to a path $\{ (B_t , C_t ) \}_{t\geq s}$ starting from $(B_s , C_s)$ such that 
$$\chi (B_t , C_t) = (b_t , c_t , k_t).$$
Moreover, $k_t < 2 $ for all $t \geq s$, and hence the path $\{ B_t C_t \}_{t\geq s}$ never meets the set $\mathcal{H}_0 \cup \mathcal{P}_0$. We infer that $B_1 C_1$ is an element of $\mathcal{H}_{-1}$, and that the representation given by $(B_1 , C_1)$ has Euler class $-1$.

We obtain $\chi (B_1, C_1) = (b,c,k) = \chi (\widetilde{\phi (B)}^0,  \widetilde{\phi (C)}^0 )$. We know that $\kappa (b,c,k) \neq 2$ and hence the two couples are conjugated by an element $g \in \pgl$. An element of $\pgl$ that is not in $\psl$ conjugates a representation in $e^{-1} (-1)$ to a representation in $e^{-1} (1)$. As each couple defines a representation with Euler class $-1$, they are conjugated by an element of $\psl$. The path 
$$ \{ (g B_t g^{-1} , g C_t g^{-1} ) \}_{t\in [0, 1]}$$
joins the representation $\phi$ with the representation $\phi '$ defined by 
\begin{align*}
\phi ' (B) & = g \phi (B) g^{-1},\\
\phi'(C) & = g (\phi(B) )^{-1} g^{-1} .
\end{align*}

Let $\psi \in W (M)$ be a representation with $e(\psi) = 1$. We can use the same arguments to prove that there exists a path joining the representations $\psi$ and $\psi'$, where $\psi'$ is defined by 
\begin{align*}
\psi ' (B) & = h \psi (B) h^{-1}, \\
\psi'(C) & = h (\psi(B) )^{-1} h^{-1}.
\end{align*}

Therefore, it suffices to find a path joining $\phi'$ to $\psi'$ that satisfies the desired properties to prove the proposition. 
The elements $ g \phi' (B) g^{-1}$ and $h \psi '(B) h^{-1}$ are in $\mathcal{H}$ and can be connected by a path $\{ H_t \}_{t\in [0,1]} $ that lies inside $\mathcal{H}$. This defines a path $\{ \phi_t ' \}_{t\in [0,1]}$ of representations from $\phi '$ to $\psi '$ by:
 \begin{align*}
\phi_t ' (B) & = H_t \\
\phi_t '(C) & = (H_t )^{-1}.
\end{align*}

We have constructed paths from $\phi$ to $\phi'$, from $\phi'$ to $\psi'$ and from $\psi'$ to $\psi$, satisfying  the condition $(3)$. Hence, we have a path $\{ \phi_t \}_{t\in [0,1]}$ that satisfies $(1) , (2) $ and $(3)$. 

Finally, let $\widetilde{\phi_t } : \pi \rightarrow \Gt$ be a lift of the path $\phi_t$ such that $\widetilde{\phi_0 } (K)$ is  in $z^\epsilon \mathcal{J}$. Let $\widetilde{\phi_t  (B)}^0$ and $\widetilde{\phi_t  (C)}^0$ be the lifts in $\mathcal{H}_0$. Then there exists $N \in \N$ such that for all $t \in [0 , 1]$, we have 
$$\widetilde{\phi_t } (K) = \widetilde{\phi_t} (B) \widetilde{\phi_t} (C) =  z^N \left( \widetilde{\phi_t  (B)}^0 \widetilde{\phi_t  (C)}^0 \right). $$
The element $\widetilde{\phi_0 } (K)$ is  in $\mathcal{H}_{N-1}$ because $e(\phi) = -1$. Hence 
$$\widetilde{\phi_0 } (K) \in z^\epsilon \mathcal{J} \cap \mathcal{H}_{N-1},$$
and thus $N$ and $\epsilon$ are of different parity.

In the construction, the path $\{ \widetilde{\phi_t (B)}^0 \widetilde{\phi_t (C)}^0 \}_{t\in [0,1]}$ does not meet any of the $\mathcal{H}_i \cup \mathcal{P}_i$ where $i$ is even. We infer that $\widetilde{\phi_t } (K)$ does not meet any of the $z^{\epsilon} \mathcal{H}_i \cup \mathcal{P}_i$ where $i$ is odd, which means $\widetilde{\phi_t } (K)$ lies within $z^\epsilon \mathcal{J}$. The other hypotheses of Proposition \ref{prop:liftsquare} are naturally satisfied by construction.

\end{proof}

For the one-holed torus, the following proposition gives a similar statement.

\begin{proposition}\label{lem:releveJ2}
Let $\Sigma$ be a one-holed torus, with fundamental group 
$$\pi_1 (\Sigma) = \langle X , Y , K | [X,Y] = K\rangle.$$ 
Let $\phi$ be a representation in $W(\Sigma)$ such that $e( \phi) = -1$. There exists a path of representations $\{ \phi_t \}_{t\in [0,1]}$ in $\Hom (\pi , G)$ such that:
\begin{enumerate}
\item $\phi_0 = \phi$;
\item $e (\phi_1) = 1$;
\item If $\{ \widetilde{\phi_t} \}_{t\in [0,1]}$ is a lift of the path in $\Hom (\pi, \Gt)$ such that \hbox{$\widetilde{\phi_0} (K) \in z^\epsilon \mathcal{J}$,}  then the path $\{z^\epsilon  \widetilde{\phi_t} (K) \}_{t\in [0,1]}$ satisfies the hypotheses of Proposition \ref{prop:liftsquare}.
\end{enumerate}
\end{proposition}
\begin{proof}
Let $\widetilde{\phi}$ be a lift of $\phi$. The element $\widetilde{\phi} (K) = \widetilde{\phi} ([X , Y])$ is uniquely determined and belongs to $\mathcal{H}_{-1}$, thus it belongs  to $z \J$. Let $K_t$ be a path starting at $K_0 = \widetilde{\phi} (K)$ and ending at an element $K_1 \in \mathcal{H}_1$ that lies in the image of the map
\begin{align*}
\widetilde{R_1} : G \times G & \longrightarrow \Gt \\
(A , B) & \longmapsto [A , B] .
\end{align*}
We can choose $K_t$ so that the path $\{ z K_t \}_{t\in [0,1]} $ satisfies the hypotheses of Proposition \ref{prop:liftsquare}. According to Goldman (\cite{goldman}, Theorem 7.1), any path in the image of $\widetilde{R_1}$ can be lifted to a path in $G\times G$. Hence, there exists a path  $\{ (X_t , Y_t) \}_{t\in [0,1]} \in G^2$ such that $\widetilde{R_1} (X_t , Y_t) = K_t$ for all $t \in [0,1]$. 

Moreover, the preimage $\widetilde{R_1} (K_0)$ is connected, and hence we can find a path joining $\phi$ to the representation defined by $(X_0 , Y_0)$. The path of representations defined by $(X_t , Y_t)$ has the desired properties. 
\end{proof}

We can establish a generalization of this proposition to any orientable surface with boundary, as follows:

\begin{proposition}\label{prop:eulermax}
Let $\Sigma$ be an orientable surface with $m\geq 1$ boundary components $C_1 , \dots , C_m$ with $\chi (\Sigma) \leq -1$. Let $\phi$ be a representation in $W (\Sigma)$ with relative Euler class $e(\phi) = n \leq -\chi(\Sigma) -2$. There is a path $\{ \phi_t \}_{t\in [0,1]}$ in $\Hom (\pi , G)$ such that:
\begin{enumerate}
\item $\phi_0 = \phi$;
\item $e (\phi_1) = n + 2$;
\item $\phi_t (C_j) \in \mathcal{H}$ for all $j > 1 $ and all $t \in [0 , 1]$;
\item If $\{ \widetilde{\phi_t} \}_{t\in [0,1]}$ is a lift of the path in $\Hom (\pi, \Gt)$ such that \hbox{$\widetilde{\phi_0} (C_1) \in z^\epsilon \mathcal{J}$,}  then  the path $\{z^\epsilon  \widetilde{\phi_t} (C_1) \}_{t\in [0,1]}$ satisfies the hypotheses of Proposition \ref{prop:liftsquare}. 
\end{enumerate} 
\end{proposition}

\begin{proof}
The result was proved for $\chi (\Sigma) = -1$, so we can suppose that $\chi (\Sigma) \leq -2$. For a surface $\Sigma$ with $\chi(\Sigma) \leq -2$,  let $\displaystyle{\Sigma = \bigcup_{i=1}^{-\chi(\Sigma)} \Sigma_i}$ be a decomposition into pair-of-pants, such that $C_1$ is a boundary component of $\Sigma_1$. Denote the two other boundary components of $\Sigma_1$ by $D_1$ and $D_2$ and by $S_i$ the connected component of $\Sigma \setminus \Sigma_1$ that contains $D_i$ as a boundary component (the surface $S_i$ can be empty). We can choose the pair-of-pants decomposition so that $S_1$ and $S_2$ are disjoint.

Let $\phi$ be a representation in $W(\Sigma)$ with relative Euler class $e(\phi) = n \leq -\chi(\Sigma) -2$. We can assume without loss of generality that the elements $\phi (D_1), \phi (D_2)$ are in $\mathcal{H}$. Moreover as $n \leq -\chi(\Sigma) -2$ and $e^{-1}  (n)$ is connected, the representation can be chosen such that the restriction $\psi_0$ of $\phi$ to $\pi_1 (\Sigma_1)$ satisfies $e(\psi_0) = -1$. 

According to Lemma \ref{lem:releve-1}, there is a path $\{ \psi_t \}_{t\in [0,1]}$ of representation in $\Hom (\pi_1 (\Sigma_1) , G)$ starting from $\psi_0$ such that:
\begin{enumerate}
\item $\psi_t (D_1)$ and $\psi_t (D_2)$ lie within $\mathcal{H}$; 
\item If $\{ \widetilde{\phi_t} \}_{t\in [0,1]}$ is a lift of the path in $\Hom (\pi, \Gt)$ such that $\widetilde{\phi_0} (C_1) \in z^\epsilon \mathcal{J}$,  then  the path $\{z^\epsilon  \widetilde{\phi_t} (C_1) \}_{t\in [0,1]}$ satisfies the hypotheses of Proposition \ref{prop:liftsquare};
\item $e(\psi_1) = 1$.
\end{enumerate}

If $S_i$ is non-empty, we can apply the Lemma \ref{lem:releve}, and find a path $\left\{ \psi_t^{(i)} \right\}_{t\in [0,1]}$ of representations in $W (S_i)$, such that $\psi_t^{(i)} (D_i) = \psi_t (D_i)$ and $\psi_0^{(i)}$ is the restriction of $\phi$ to $\pi_1 (S_i)$.
We define the path $\{ \phi_t \}_{t\in [0,1]}$ by: 
$$\phi_t (\gamma) = \left\{ \begin{array}{ll} \psi_t (\gamma) & \mbox{ if } \gamma \in \pi_1 (\Sigma_1) \\ \psi_t^{(i)} (\gamma) & \mbox{ if } \gamma \in \pi_1 (S_i) . \end{array} \right.$$

We have $\phi_0 = \phi$ by construction. 

For $j>1$, if $C_j$ is a boundary component of $S_i$, then $\phi_t (C_j) = \psi_t^{(i)} (C_j) \in \mathcal{H}$, because $\psi_t^{(i)} \in W(S_i)$. Otherwise $C_j$ is a boundary component of $\Sigma_1$ and the Lemma \ref{lem:releve-1} gives us that $\phi_t (C_j) \in \mathcal{H}$.

We have $e(\psi_1) = 1 = e(\psi_0) +2$,  and $\psi_1^{(i)}$ is in the same connected component of $W(S_i)$ than $\psi_0^{(i)}$. Therefore $e(\psi_1^{(i)}) = e(\psi_0^{(i)})$.

The additivity formula of the Euler class gives us 
$$e (\phi_1) = e(\psi_1) + \sum_i e(\psi_1^{(i)}) = e(\psi_0) +2 + \sum_i e(\psi_0^{(i)}) = e(\phi_0) +2 .$$

The construction of the path $\{ \psi_t \}_{t\in [0,1]}$ insures that for any lift $\widetilde{\phi_t}$ in $\Hom (\pi , \Gt)$, the path $\widetilde{\phi_t} (C_1) = \widetilde{\psi_t } (C_1)$ has the required properties.
\end{proof}


\section{Connected components}\label{section:proof}

A non-orientable surface is the connected sum of an orientable surface $\Sigma$ with one or two projective planes. The idea is to consider representations whose restriction to the orientable surface $\Sigma$ is in $W (\Sigma)$.

\subsection{Surface decomposition}


The following lemma is a direct consequence of Lemma \ref{lem:pathliftchi}

\begin{lemma}\label{lem:Khyp} Let $M $ be a two-holed projective plane with fundamental group 
$$\pi = \langle A , B , C , K | A^2 = K = BC \rangle.$$
If $\phi $ is a representation in $\Hom (\pi, G)$ so that $\phi (C) \in \mathcal{H}$, then there exists a path $\{ \phi_t \}_{t\in [0,1]}$ in $\Hom (\pi, G)$ such that :
\begin{enumerate}
\item $\phi_0 = \phi$;
\item $\phi_t (B)$ is conjugated to $\phi(B)$ and $\phi_t (C)$ is conjugated to $\phi (C)$, for all $t$;
\item $ \phi_1 (K)$ is hyperbolic.
\end{enumerate}
\end{lemma}
\begin{proof}
Let  $\widetilde{\phi} : \pi \rightarrow \Gt$ be a lift of the representation $\phi$. Let $b,c,k$ denote the traces of the elements $\widetilde{\phi}(B)$, $\widetilde{\phi}(C)$ and $\widetilde{\phi}(K)$. If $\phi(K)$ is already hyperbolic there is nothing to prove, thus we assume that $-2 \leq k \leq 2$. For $\epsilon > 0$, consider the path $\{ ( b_t, c_t , k_t) \}_{t\in [0,1]}$ defined by
$$b_t = b , \hspace{1cm} c_t = c, \hspace{1cm} k_t = t (2 +\epsilon) + (1-t)k . $$
As $|c| > 2$, this path never meets the set $[-2 , 2]^3 \cap \kappa^{-1} ([-2 , 2])$, and hence can be lifted to a path $(B_t , C_t) \in \Gt \times \Gt$ starting from $(B_0 , C_0)= (\widetilde{\phi}(B), \widetilde{\phi}(C))$.

The element $K_0 = \widetilde{\phi} (K) = (\widetilde{\phi}(A))^2$ belongs to $\J$. Moreover, the path $K_t = B_t C_t$ satisfies $k_t > -2$, and $[B_t , C_t ]\neq I$ for all $t > 0$. Hence the path $K_t$ lies entirely within $\J \setminus Z(\Gt)$ for all $t > 0$. 
This proves the existence of a path $A_t \in \Gt$ such that $A_t^2 = K_t$. 

It follows that $\left\{ ( A_t , B_t , C_t , K_t ) \right\}_{t\in [0,1]}$ defines a path with the desired properties.
\end{proof}

\begin{lemma}\label{lem:KAhyp}
Let $M$ be a non-orientable surface with $\chi (M) \leq -1$ and $N$ an embedded M\"obius strip inside $M$. Let $S$ be the subsurface $M \setminus N$ and $K$ the common boundary $\partial S \cap \partial N$.

Let $\phi$ be a representation in $W( M)$. There exists a path $\{ \phi_t \}_{t\in [0,1]}$ such that $\phi_1 (K)$ is hyperbolic.
\end{lemma}
\begin{proof}
First, assume that $\partial M \neq \emptyset$ and let $C \subset \partial M$. Then there is an embedded three-holed sphere $\Sigma_1$ in $M$, having $C$ and $K$ as boundary component. Denote by $B$ its third boundary component. The Lemma \ref{lem:Khyp} applies to the restriction $\psi$ of $\phi$ to the surface $N \cup \Sigma_1$. Hence there exists a path $\{ \psi_t \}_{t \in [0,1]}$ such that $\psi_1 (K)$ is hyperbolic. Moreover, there exists a path $\{ U_t \}_{t \in [0,1]}$ in $\Gt$ such that $\psi_t (B) = U_t \phi (B) U_t^{-1}$. The path $\{ \phi_t \}_{t\in [0,1]}$ given by:
$$
\begin{array}{ll}
\phi_t (\gamma) = \psi_t (\gamma),  \hspace{1cm} &\mbox{if } \gamma \in \pi_1 (N \cup \Sigma_1) ,\\
\phi_t (\gamma) = U_t \phi (\gamma) U_t^{-1}, & \mbox{if } \gamma \in \pi_1 (S),
\end{array}
$$
has the desired properties.

Now, if $M$ is a closed surface. Then $S$ is a one-holed torus or $\chi (S) \leq -2$. 

If $S$ is a one-holed torus, let $\psi$ be the restriction of $\phi$ to $\pi_1 (S)$, and $\widetilde{\psi}$ a lift of this representation. Using the same arguments as in the proof of Proposition \ref{lem:releveJ2}, there exists a path $\{ \psi_t \}_{t \in [0,1]}$ joining $\psi$ to a representation $\psi'$ such that $\psi' (K)$ is hyperbolic. This path has the property that whenever $z^\epsilon \widetilde{\psi}_0 (K)$ is in $\J$, the path $\left\{ z^\epsilon \widetilde{\psi_t} (K) \right\}_{t \in [0,1]}$ lies within $\J$. Hence, there exists a path $\{ A_t \}_{t \in [0,1]}$ in $\Gt$ such that $A_t = z^\epsilon  \widetilde{\psi_t} (K)$ and this defines a path of representations $\{ \phi_t \}_{t \in [0,1]} $ in $\Hom (\pi , G)$ with the desired properties.

Otherwise, we have $\chi(S) \leq -2$. There is an embedded three-holed sphere $\Sigma_1$ in $S$ having $K$ as boundary component, such that $S \setminus \Sigma_1)$ has two connected components, one of them being a one-holed torus $T$. Denote by $B$ and $C$ the other boundary components of $\Sigma_1$. The surface $\Sigma = T \cup \Sigma_1$ is a two-holed torus with boundary components $K$ and $B$. There exists a path $\{ \psi_t \}_{t\in [0,1]}$ in $\Hom (\pi_1 (\Sigma) , G)$ starting from $\psi_0 = \phi_{|\pi_1 (\Sigma) }$ such that $\phi_1 (C)$ is hyperbolic (see Goldman \cite{goldman} Lemma 9.3). Moreover there exists path $\{ U_t \}_{t\in [0,1]}$ and $\{ V_t \}_{t\in [0,1]}$ in $\Gt$ such that 
$$\psi_t (K) = U_t \phi (K) U_t^{-1}, \hspace{1cm} \psi_t (B) = V_t \phi (B) V_t^{-1}.$$ 
The path defined by
$$
\begin{array}{ll}
\phi_t (\gamma) = \psi_t (\gamma)   \hspace{1cm}&\mbox{if } \gamma \in \Sigma ,\\
\phi_t (\gamma) = U_t \phi (\gamma) U_t^{-1}  &\mbox{if  }\gamma \in N , \\
 \phi_t (\gamma) = V_t \phi (\gamma) V_t^{-1}  &\mbox{if  }\gamma \in M\setminus (N \cup \Sigma) , 
\end{array}
$$
is such that $\phi_1 (C)$ is hyperbolic. Now we can apply Lemma \ref{lem:Khyp} to the restriction of the representation $\phi_1$ to $\pi_1 (N \cup \Sigma_1)$ and apply the same argument as in the case of an open surface.
\end{proof}


\subsection{Proof of Theorem \ref{thm:open}}

We resume now the proof of Theorem \ref{thm:open}.

First, let $M$ be a non-orientable surface of odd genus $k$. There is an embedded M\"obius strip $N$ inside $M$  such that the surface $S = M \setminus N$ is an orientable surface with Euler characteristic $\chi(S)  = \chi (M)$. Let $\phi$ and $\psi$ be two representations in $W(M)$ such that $o_2 (\phi) = o_2 (\psi)$.  We denote by $\phi '$ and $\psi '$ the restrictions to $\pi_1 (S)$ of the representations $\phi$ and $\psi$. According to Proposition \ref{lem:KAhyp}, we can assume that $\phi (K)$ and $\psi(K)$ are hyperbolic, and hence $\phi'$ and $\psi'$ are in $W(S)$.

We infer from Proposition \ref{prop:eulermax} that we can choose $\phi$ and $\psi$ such that the relative Euler classes $e(\phi ')$ and $e(\psi ')$ are in $\{ -\chi (S) , -\chi (S) -1 \}$. Indeed, we can find a path of representations $\{ \phi_t ' \}_{t\in [0,1]}$ in $\Hom (\pi_1 (S) , G)$ such that $e (\phi_1 ') \in \{ -\chi (S) , -\chi (S) -1 \}$ and $\phi_t '(K) \in \J$. Hence, according to Proposition \ref{prop:liftsquare}, we can lift the path of representations $\{ \phi_t' \}_{t\in [0,1]}$ to a path $\{ \phi_t \}_{t\in [0,1]}$ of representations in $W(M)$.

A presentation of the fundamental group is given by:
$$\pi = \left\langle A , X_1 , \dots , Y_g ,  C_1 , \dots , C_m \, | \, A^2 [X_1 , Y_1] \cdots [X_g , Y_g] C_1 \cdots C_m \right\rangle $$
where the element $A$ satisfies $A^2 = K$. Choose lifts $\widetilde{\phi (A)}, \widetilde{\phi (X_1)} , \dots , \widetilde{\phi (Y_g)}$ in $\Gt$ and let $\widetilde{\phi (C_1)}^0 , \dots , \widetilde{\phi (C_m)}^0  $ be the lifts in $\mathcal{H}_0$. By definition of the map $o_2$, we have
$$ \left( \widetilde{\phi (A)} \right)^2 \left[ \widetilde{\phi (X_1)} , \widetilde{\phi (Y_1)} \right]  \cdots \left[\widetilde{\phi (X_g)},  \widetilde{\phi (Y_g)} \right] \widetilde{\phi (C_1)}^0 \cdots \widetilde{\phi (C_m)}^0 = z^{2n + o_2 (\phi)}$$
for some $n \in \Z$.

On the other hand, if $\widetilde{\phi (K)}^0$ is the lift of $\phi (K)$ in $\mathcal{H}_0$ then by definition of the Euler class we have:
$$ \widetilde{\phi (K)}^0 \left[ \widetilde{\phi (X_1)} , \widetilde{\phi (Y_1)} \right]  \cdots \left[ \widetilde{\phi (X_g)},  \widetilde{\phi (Y_g)} \right] \widetilde{\phi (C_1)}^0 \cdots \widetilde{\phi (C_m)}^0 = z^{e(\phi ')}.$$

Hence, we obtain
$$(\widetilde{\phi (A)})^2 = z^{N} \widetilde{\phi (K)}^0$$
with $N =e(\phi ') - o_2 (\phi) - 2n $. This element belongs to $\J \cap \mathcal{H}_{N}$, thus $N$ is even. This implies that $\bar{e}(\phi ') = o_2 (\phi)$ in $\Z /2 \Z$, where $\bar{ e} = e (\mbox{mod } 2)$. The same applies to $\psi$ and hence $\bar{e}(\psi) = o_2 (\psi)$ in $\Z /2 \Z$. We infer that $e(\phi ') = e(\psi ')$, which implies that $\phi'$ and $\psi'$ are in the same connected component of $W(S)$.

Therefore there exists a path $\{ \phi_t ' \}_{t\in [0,1]}$ joining $\phi '$ to $\psi '$ in $W(S)$. The fact that we can find a path joining $\phi$ to $\psi$ in $W(M)$ is a consequence of the following Lemma:

\begin{lemma}\label{lem:extension}
Let $M$ be a non-orientable surface with $\chi (M) \leq -1$. Let $N \subset M $ be an embedded Möbius strip and $S = M \setminus N$. Let $\phi$ and $\psi$ be representations in $W(M)$, and let $\phi '$ and $\psi '$ be their restrictions to $\pi_1 (S)$. 

If $\phi '$ and $\psi '$ are in $W(S)$, then a path joining $\phi '$ and $\psi '$ in $W(S)$ can be extended to a path joining $\phi$ and $\psi$ in $W(M)$.
\end{lemma}
\begin{proof}
Let $A$ be the generator of $\pi_1 (N) \hookrightarrow \pi_1 (M)$, so that $A^2 = K$ is the homotopy class of the common boundary of $N$ and $S$. In particular the condition $\phi' \in W(S)$ implies that $\phi' (K)$ is hyperbolic. Let $\{ \phi_t ' \}_{t\in [0,1]}$ be the path joining $\phi '$ and $\psi '$ in $W(S)$. There exists $A_0$ in $G$ such that $\phi_0 ' (K) = A_0^2$. Moreover for any $t \in [0 , 1]$, the element $\phi_t ' (K)$ is hyperbolic. Hence according to Proposition \ref{prop:square}, for any $t \in [0,1]$ there exists a unique element $A_t$ such that $A_t^2 = \phi_t ' (K)$. We have the unicity of $A_t$ and this proves that $A_0 = \phi (A)$ and $A_1 = \psi (A)$. Thus, we can define the path $\phi_t$ in $W(M)$ by:
\begin{align*}
\phi_t (\gamma) & = \phi_t ' (\gamma) \mbox{ if } \gamma \in \pi_1 (S),\\
\phi_t (A) & = A_t .
\end{align*}
This path joins $\phi$ to $\psi$ in $W(M)$.
\end{proof}

\medskip

Now, let $M$ be a non-orientable surface of even genus $k$ and $N$ an embedded M\"obius strip inside $M$ with boundary $K$. The surface $S = M \setminus N$ is a non-orientable surface of odd genus. Let $\phi$ and $\psi$ be two representations in $W(M)$ such that $o_2 (\phi) = o_2 (\psi)$. According to Lemma \ref{lem:KAhyp}, we can assume that $\phi (K)$ and $\psi(K)$ are hyperbolic. 

We consider the restrictions $\phi '$ and $\psi '$ of $\phi$ and $\psi$ to $\pi_1 (S)$. The previous case ensures that we can find a path $\{ \phi_t ' \}_{t\in [0,1]}$ in $W(S)$ joining $\phi'$ to $\psi'$. Then, Lemma \ref{lem:extension} gives us a path joining $\phi$ to $\psi$ in $W(M)$ which concludes the proof.

\end{document}